\documentclass[12pt]{amsart}
\usepackage{amssymb,amsfonts,amsmath}
\usepackage{mathrsfs}
\usepackage{anysize}
\usepackage{url}
\usepackage[leqno]{amsmath}
\usepackage{enumitem}
\usepackage{tikz-cd}
\usepackage{stackengine}
\usetikzlibrary{decorations.pathreplacing}

\theoremstyle{plain}
\newtheorem{theorem}{Theorem}[section]
\newtheorem{lemma}[theorem]{Lemma}

\theoremstyle{definition}
\newtheorem{definition}[theorem]{Definition}
\theoremstyle{remark}
\newtheorem{remark}[theorem]{Remark}
\newtheorem{example}[theorem]{Example}
\DeclareRobustCommand{\rchi}{{\mathpalette\irchi\relax}}
\newcommand{\irchi}[2]{\raisebox{\depth}{$#1\chi$}} 

\usetikzlibrary{decorations.markings}
\tikzset{negated/.style={
        decoration={markings,
            mark= at position 0.5 with {
                \node[transform shape] (tempnode) {${\scriptstyle\setminus} $};
            }
        },
        postaction={decorate}
    }
}

\tikzset{degil/.style={
            decoration={markings,
            mark= at position 0.5 with {
                  \node[transform shape] (tempnode) {$\backslash$};
                  }
              },
              postaction={decorate}
}
}

\title[Piecewise contractions and $b$-adic expansions]{Piecewise contractions and $b$-adic expansions}


\medskip
\begin{document}

\maketitle

\centerline{\scshape Benito Pires}

{\footnotesize
	\centerline{Departamento de Computa\c c\~ao e Matem\'atica, Faculdade de Filosofia, Ci\^encias e Letras}
	\centerline {Universidade de S\~ao Paulo, 14040-901, Ribeir\~ao Preto - SP, Brazil}
	\centerline{benito@usp.br} }

\marginsize{2.5cm}{2.5cm}{1cm}{2cm}

\begin{abstract} Let $I=[0,1)$, $b\in \{2,3,\ldots\}$ and $f:I\to I$ be an injective piecewise $\frac{1}{b}$-affine map, that is, assume that there exists a partition of $I$ into intervals $I_1,\ldots,I_n$ such that $\vert f(x)-f(y)\vert\le\frac1b \vert x-y\vert$ for all $x,y\in I_i$ and $1\le i\le n$. In this note, we study the $\delta$-parameter family of maps $f_{\delta}=R_{\delta}\circ f$, where $R_\delta:x\mapsto \{x+\delta\}$. More precisely, we show that the set $\mathcal{N}$ of parameters $\delta$ for which $f_{\delta}$ has only natural codings with maximal complexity is a non-empty set with Hausdorff \mbox{dimension $0$}. We also show that for all $\delta\in\mathcal{N}$, the map $f_{\delta}$ is topologically semiconjugate to a minimal $n$-interval exchange transformation satisfying Keane's i.d.o.c. condition. The main result turns out to be a concrete application of the result by Mauduit and Moreira that the set of numbers having $b$-adic expansion with entropy $0$ has Hausdorff dimension $0$.   
 \end{abstract}

\maketitle

\section{Introduction}

Let $I=[0,1)$ denote the unit interval. A map $f:I\to I$ is a {\it piecewise contraction of $n$ intervals} ($n$-PC) if there are $0<\lambda<1$, $n\ge 2$ and 
$0=x_0<x_1<\ldots<x_{n-1}<x_n=1$ such that $\left\vert f(x)-f(y)\vert\le \lambda\vert x-y\right\vert$ for all $x,y \in [x_{i-1},x_i)$ and $1\le i\le n$. If there exist $d_1,\ldots,d_n\in\mathbb{R}$ such that $f(x)=\lambda x +d_i$ for all $x\in \left[x_{i-1},x_i\right)$, then we call $f$ a \textit{piecewise $\lambda$-affine map}. We assume that $\mathcal{D}(f)=\{x_1,\ldots,x_{n-1}\}$ is the discontinuity set of $f$. The set $G=I{\setminus}f(I)$ is called the \textit{gap set}. An infinite word $\theta=\theta_0 \theta_1\ldots$ over the alphabet $\{1,\ldots,n\}$ is the {\it natural $f$-coding} of $x\in I$ if $\theta_i=\epsilon\big(f^i(x)\big)$ for all $i\ge 0$, where $\epsilon:I\to \{1,\ldots,n\}$ is defined by $\epsilon(y)=i$ if $y\in [x_{i-1},x_i)$.
We denote by $L_k(\theta)=\{\theta_{\ell} \theta_{\ell+1}\ldots \theta_{\ell+k-1}: \ell\ge 0\}$ the set of subwords of $\theta$ of length $k$. The \textit{complexity function} of $\theta$ is the map 
$p_\theta:\mathbb{N}\to\mathbb{N}$ defined by $p_\theta(k)=\#L_k(\theta)$, where $\#$ stands for cardinality. It is known (see \cite{CGM} or \cite[Corollary $2.4.(i)$]{BP}) that if $\theta$ is any natural coding of 
 an injective $n$-PC, then there exists $k_0\in\mathbb{N}$ such that $p_\theta(k)\le (n-1)k +1$ for all $k\ge k_0$.
 The aim of this note is to study the set $\mathscr{C}$  of $n$-PCs whose natural codings $\theta$ have the maximal complexity function $p_\theta(k)=(n-1)k+1$ for all $k\ge 1$. By \cite[Corollary $2.4.(ii)$]{BP} and by Keane's Irrationality Criterium, the set $\mathscr{C}$ is at least as big as the set of $n$-IETs with irrational length vectors. 
 
 We will need the following definition.
 
 \begin{definition} An $n$-PC $f:I\to I$ \textit{has only natural codings with maximal complexity} if
 every natural coding $\theta$ of $f$ has the complexity function $p_\theta(k)=(n-1)k+1$ for all $k\ge 1$.
 \end{definition}
Notice that when $n=2$, a $2$-PC $f:I\to I$ has only natural codings with maximal complexity if and only if every natural coding of $f$ is a Sturmian word. 

 The relation between piecewise contractions and $b$-adic expansions of real numbers is revealed when we consider piecewise $\lambda$-affine maps with $\lambda=\frac{1}{b}$. This is the case of the one-parameter family of $n$-PCs defined next.  As usual, $R_{\delta}:I\to I$ denotes the rotation by $\delta$, that is, $R_{\delta}(x)=\{x+\delta\}$ for all $x\in I$.
 
  \begin{definition}[The family of maps $f_{\delta}$]\label{def1} Given integers  $b\ge n\ge 2$ and real numbers $0=x_0<x_1<\ldots<x_{n-1}<x_n=1$,  
  let $f:I\to I$ be an injective piecewise $\dfrac1b$-affine map defined by $$f(x)=\dfrac1b x+ d_i\quad \textrm{for all}\quad x\in \left[x_{i-1},x_i\right)\quad\mathrm{and}\quad 1\le i\le n,$$
  where $d_1,\ldots,d_n\in\mathbb{R}$ are such that $d_{i}-d_{i+1}=1$ for some $1\le i\le n-1$. We assume that the connected components $G_1$,\ldots, $G_{m}$ of $G:=I{\setminus}f(I)$ have the same length $L$. For each $\delta\in\mathbb{R}$, let $f_{\delta}:I\to I$ denote the map
\begin{equation}\label{eq1}
f_{\delta}(x)=\left(R_{\delta}\circ f\right)(x)\quad\textrm{for all}\quad x\in I.\end{equation}
 \end{definition}
 \begin{remark}\label{rem} In Definition \ref{def1}, the condition $d_{i}-d_{i+1}=1$ means that the interval map $f:I\to I$ is induced by a piecewise affine self-map of the unit circle $\widehat{f}:\mathbb{T}^1\to\mathbb{T}^1$. In fact, such condition is equivalent to the following conditions:
 $\lim_{\epsilon\to 0^+}f(x_i-\epsilon)=1$ and $f(x_i)=0$.
 \end{remark}
 \begin{example} The maps $f:I\to I$ defined below satisfy the hypotheses of Definition \ref{def1}.
 \begin{equation*}
(a)\,\,f(x)=\begin{cases} \dfrac12 x+\dfrac23 & \mathrm{if}\quad x\in \left[0,\dfrac23\right) \\[0.2in]
\dfrac12 x-\dfrac13 & \mathrm{if}\quad x\in \left[\dfrac23,1\right) 
\end{cases},\quad
(b)\,\,f(x)=\begin{cases} \dfrac14 x+\dfrac{17}{40} & \mathrm{if}\quad x\in \left[0,\dfrac25\right) \\[0.2in]
\dfrac14 x+\dfrac45 & \mathrm{if}\quad x\in \left[\dfrac25,\dfrac45\right)\\[0.2in]
\dfrac14 x-\dfrac{1}{5}& \mathrm{if}\quad x\in \left[\dfrac45,1\right)
\end{cases}.
\end{equation*}
\end{example}
The corresponding plots are given in Figure \ref{fig1}.
  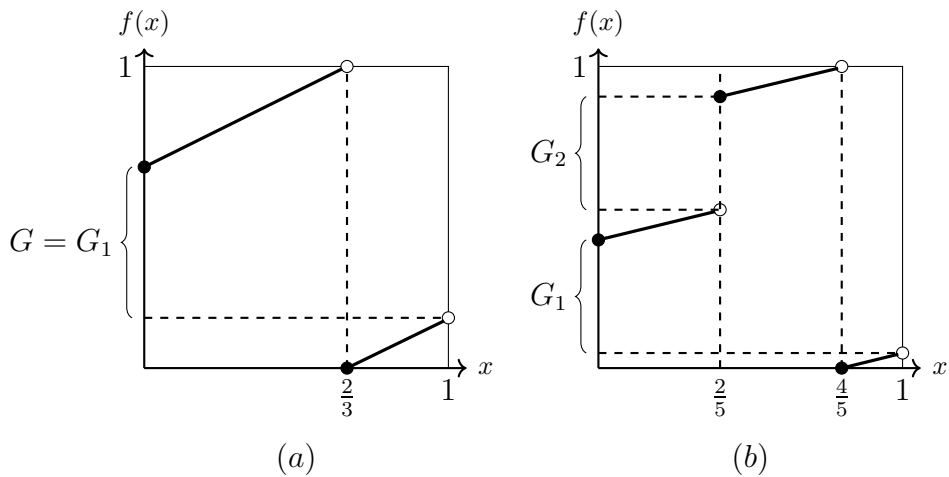
\begin{figure}[!htb]
     \begin {minipage}{0.4\textwidth}
     \centering  
\begin{tikzpicture}[scale=0.8]



\draw [  thick, ->] (0,0) -- (5.3,0) node [right] {\footnotesize $x$};
\draw [  thick, ->] (0,0) -- (0,5.3) node [above] {\footnotesize $f(x)$};
	
			

\draw (0,0)--(5,0)--(5,5)--(0,5);
			
\draw[fill=black] (0,10/3) circle (0.1);
\draw[fill=white] (10/3,5) circle (0.1);
\draw[very thick] (0,10/3)--(10/3-0.08,5-0.04) ;

\draw[fill=black] (10/3,0) circle (0.1);
\draw[fill=white] (5,5/6) circle (0.1);
\draw[very thick] (10/3,0)--(5-0.07,5/6-0.05) ;

\draw [  thick, dashed] (10/3,0)--(10/3,5);

\draw (10/3,0) node[below] {$\frac23$};
\draw (5,0) node[below] {$1$};
\draw (0,5) node[left] {$1$};

\draw [  thick, dashed] (0,5/6)--(5,5/6);

 \draw [decorate,decoration={brace,amplitude=4pt},xshift=-0.2cm,yshift=0pt]
      (0,5/6) -- (0,10/3) node [midway,left,xshift=-.1cm] {$G=G_1$};

\draw[very thick,dashed] node at (2.5,-1.5) {$(a)$};
 
\end{tikzpicture}
   \end{minipage}
    \begin{minipage}{0.4\textwidth}
     \centering
 
\begin{tikzpicture}[scale=0.8]



\draw [  thick, ->] (0,0) -- (5.3,0) node [right] {\footnotesize $x$};
\draw [  thick, ->] (0,0) -- (0,5.3) node [above] {\footnotesize $f(x)$};
	
			

\draw (0,0)--(5,0)--(5,5)--(0,5);
			
\draw[fill=black] (0,17/8) circle (0.1);
\draw[fill=white] (2,1/2+17/8) circle (0.1);
\draw[very thick] (0,17/8)--(2-0.08,1/2+17/8-0.03) ;

\draw[fill=black] (2,4.5) circle (0.1);
\draw[fill=white] (4,5) circle (0.1);
\draw[very thick] (2,4.5)--(4-0.06,5-0.03) ;

\draw[fill=black] (4,0) circle (0.1);
\draw[fill=white] (5,5/20) circle (0.1);
\draw[very thick] (4,0)--(5-0.07,5/20-0.02) ;

\draw (5,0) node[below] {$1$};
\draw (0,5) node[left] {$1$};
\draw (4,0) node[below] {$\frac45$};

\draw [  thick, dashed] (2,0)--(2,5);
\draw [  thick, dashed] (4,0)--(4,5);
\draw [  thick, dashed] (0,1/2+17/8)--(2,1/2+17/8);
\draw [  thick, dashed] (0,4.5)--(2,4.5);
\draw [  thick, dashed] (0,5/20)--(5,5/20);

 \draw [decorate,decoration={brace,amplitude=4pt},xshift=-0.2cm,yshift=0pt]
      (0,5/20) -- (0,17/8) node [midway,left,xshift=-.1cm] {$G_1$};
      
       \draw [decorate,decoration={brace,amplitude=4pt},xshift=-0.2cm,yshift=0pt]
      (0,1/2+17/8) -- (0,4.5) node [midway,left,xshift=-.1cm] {$G_2$};

\draw (2,0) node[below ] {$\frac25$};
\draw[very thick,dashed] node at (2.5,-1.5) {$(b)$};
 
\end{tikzpicture}
   \end{minipage}\hfill
   \caption{Examples of functions $f$ satisfying Definition \ref{def1}}\label{fig1}
\end{figure}

 Much study was devoted to the family $f_{\delta}$ when $n=2$ and $b\in (1,\infty)$, buy using an approach based on Rotation Number Theory. In particular, by Bugeaud \cite{B1,B2} and by Bugeaud and Conze \cite{BC} we know that
 the set $\mathcal{N}$ of parameters $\delta$ for which $f_{\delta}$ has only natural codings of maximal complexity has Lebesgue measure $0$. More recently, Laurent and Nogueira \cite{LNo} proved that $\mathcal{N}$ has Hausdorff dimension $0$. Janson and Orbeg \cite{JO} improved that result even further: they
supplied gauge functions that provide finer information, including both an upper and a lower bound on the ``size" of $\mathcal{N}$. Concerning $n\ge 2$ and $b\in (1,\infty)$, the author, Nogueira and Rosales \cite{NPR1} showed that the set $\mathcal{N} $ has Lebesgue measure $0$. In this note, we show that for all integers $b\ge n\ge 2$, the set $\mathcal{N}$ has Hausdorff dimension $0$.\\ 

\noindent\noindent\textbf{Acknowledgments.} The author is very grateful to the help and support from the following people: Christian Mauduit for explaining part of the results in \cite{MM}; Udayan Darji for  fruitful discussions; and the organizers of the IV Workshop on Dynamics, Numeration and Tilings at Florian\'opolis, where the author got some of the insights needed to conclude this note. The author was partially supported by grant {\#}2018/06916-0, S\~ao Paulo Research Foundation (FAPESP) and by the National Council for Scientific and Technological Development (CNPq).

\section{Statement of the results}

Our first result concerns the family of maps $f_{\delta}$ introduced in Definition \ref{def1}. We recall that the set $G=I{\setminus}f(I)$ is called the gap set of the map $f$. It follows from Remark \ref{rem} that
$0<\inf G<\sup G<1$ and that $f_{\delta}$ has $n-1$ discontinuities for all $\delta\in  (-\inf G,1-\sup G)$. 

Now we state our first result.

\begin{theorem}\label{thm1} The set $\mathcal{N}\subset (-\inf G,1-\sup G)$ of parameters $\delta$ for which $f_{\delta}=R_{\delta}\circ f$ has only natural codings with maximal complexity has Hausdorff dimension $0$. \end{theorem}

The approach we use to prove Theorem \ref{thm1} is based on Number Theory. The ideia is to characterize the numbers in the set $\mathcal{N}$ in terms of their $b$-adic expansions. Then we use results obtained by Mauduit and Moreira \cite{MM} that give the Hausdorff dimension of sets of numbers with $b$-adic expansions of entropy zero.

We will need some more definitions in order to state our second result. A bijection $T:I\to I$ is an \textit{$n$-Interval Exchange Transformation} ($n$-IET) if there exist numbers
$0=y_0<y_1<\ldots<y_{n-1}<y_n=1$ and $c_1,\ldots,c_{n}\in\mathbb{R}$ such that $T(y)=y+c_i$ for all $y\in \left[y_{i-1},y_i\right)$ and $1\le i\le n$, where $\mathcal{D}(T)=\{y_1,\ldots,y_{n-1}\}$ is the discontinuity set of $T$.
We say that $T$ satisfies the \textit{i.d.o.c.} (infinite distinct orbit condition) if $O_T(y_i)=\{y_i,T(y_i),\ldots\}$, $1\le i\le n-1$, are infinite pairwise disjoint sets. Keane \cite{MK} proved that if $T$ satisfies the i.d.o.c., then all orbits of $T$ are dense. 

Our second result clarifies what does it mean for an $n$-PC to have natural codings with maximal complexity.

 \begin{theorem}[Structure Theorem]\label{thm2} Let $n\ge 2$ and $g:I\to I$ be an injective $n$-PC with discontinuity set $\mathcal{D}(g)=\{x_1,\ldots,x_{n-1}\}$. The following statements are equivalent:
 \begin{itemize}
 \item [$(i)$] $g$ has only natural codings with maximal complexity;
 \item [$(ii)$] $g$ is topologically semiconjugate to an $n$-IET satisfying the i.d.o.c.
 \end{itemize}
 \end{theorem}
 
 It is worth mentioning that Ferenczi and Zamboni \cite{FZ} studied \mbox{$n$-IETs} whose natural codings $\theta$ have complexity function $p_{\theta}(k)=(n-1)k+1$ for every $k\ge 1$, which is the case of $n$-IETs
   satisfying the i.d.o.c.

\section{Preparatory Lemmas and Proof of Theorem \ref{thm1}}

In what follows, let $n\ge 2$ and $g:I\to I$ be an injective $n$-PC with discontinuity set $\mathcal{D}(g)=\{x_1,\ldots,x_{n-1}\}$ and gap set $G=I{\setminus} g(I)$. We say that $g$ has a \textit{connection} if there exist $k\ge 1$ and
$1\le i,j\le n$ such that $g^k(x_i)=x_j$. 

\begin{lemma}\label{lem2.1} The following statements are equivalent:
\begin{itemize}
 \item [$(i)$] $g$ has only natural codings with maximal complexity;
 \item [$(ii)$] $g$ has no connection and $\mathcal{D}(g)\cap \bigcup_{k= 0}^\infty g^k(G)=\emptyset$.
 \end{itemize}
\end{lemma}
\begin{proof} We will show that $(i)\iff (ii)$.  Let $\mathcal{A}=\{1,\ldots,n\}$ and
$\mathscr{P}=\{I_1,\ldots,I_n\}$ be the partition associated with $g$, where $I_i=\left[x_{i-1},x_i\right)$. 
It is elementary that $g^{-1}(I_i)$, $1\le i\le n$, is either empty or the union of finitely many intervals with non-empty interiors. Moreover, $g^{-k}(\mathscr{P})$ is a partition of $I$ for every $k\ge 0$, implying that each member of the family $$ \mathscr{P}_k=\bigwedge_{\ell=0}^{k-1} g^{-\ell}(\mathscr{P})=\left\{I_{i_0}\cap g^{-1}\big(I_{i_1}\big)\cap \cdots\cap g^{-(k-1)}\big(I_{i_{k-1}}\big): i_0,i_1,\ldots,i_{k-1}\in\mathcal{A}\right\}
 $$
is either empty or the union of pairwise disjoint intervals. Moreover, if $\theta$ is a natural coding of $g$, then the $k$-word 
$i_0 i_1\ldots i_{k-1}$ occurs in $\theta$ if and only if $J=I_{i_0}\cap g^{-1}\big(I_{i_1}\big)\cap \cdots\cap g^{-(k-1)}\big(I_{i_{k-1}}\big)\in \mathscr{P}_k$ is non-empty. Therefore, the number of non-empty members in $\mathscr{P}_k$
equals $p_\theta(k)$ and is bounded by the number of partition points of $\mathscr{P}_k$ as follows
\begin{equation}\label{pk}
 p_\theta(k)=1+\sum_{\ell=0}^{k-1} n_{\ell}, 
  \end{equation}
 where $n_0=n-1$ and
$$
  n_{\ell}=\#\left( g^{-\ell}(\{x_1\})\cup\ldots\cup g^{-\ell}(\{x_{n-1}\}){\bigg\backslash} \bigcup_{p=0}^{\ell-1} g^{-p}(\{x_1\})\cup\ldots\cup g^{-p}(\{x_{n-1}\})\right)
  $$
  gives the number of new division points at the $\ell$-th step towards the construction of $\mathscr{P}_k$. Notice that
  $g^{-p}(\{x_i\})$ is either the empty-set or a singleton. In this way, $(i)$ occurs, i.e., 
   $p_\theta(k)=(n-1)k+1$ for all $k\ge 1$, if and only if
  $(ii)$ occurs.
\end{proof}

\begin{lemma}\label{lem2.2} Assume that $g$ has only natural codings with maximal complexity and
write its gap $G=G^{(1)}\cup \ldots\cup G^{(m)}$ as the union of its connected components. Then, the collection $\left\{g^k(G^{(j)}):1\le j\le m, k\ge 0\right\}$ is formed by pairwise disjoint intervals and 
$$\mathrm{Leb}\left(\bigcup_{j=1}^m\bigcup_{k=0}^\infty g^k(G^{(j)})\right)=1,$$
where Leb means Lebesgue measure.
\end{lemma}
\begin{proof} By Lemma \ref{lem2.1}.(ii), we have that $g^k(G^{(j)})$ is an interval for all $1\le j\le m$ and $k\ge 0$. Since $g^{-1}(G)=\emptyset$ and $g$ is injective, we have that $\left\{g^k(G^{(j)}):1\le j\le m, k\ge 0\right\}$ is formed by pairwise disjoint intervals. Let $0<c<1$ be the contraction rate of $g$ and
 $A=\bigcup_{j=1}^m\bigcup_{k=0}^\infty g^k(G^{(j)})$. We have that $I{\setminus}A=\bigcap_{k\ge 0} g^k(I)$, thus
 $\mathrm{Leb}\,(I{\setminus}A)\le c^k \mathrm{Leb}\,(I)$ for any $k\ge 0$, which yields $\mathrm{Leb}\,(A)=1$.
\end{proof}


\begin{proof}[Proof of Theorem \ref{thm2}] $(i)\implies (ii)$. Let $\omega(x)=\bigcap_{\ell=0}^\infty \overline{\bigcup_{k=\ell}^\infty \left\{g^k(x)\right\}}$ denote the $\omega$-limit set of $x\in I$. By the Morse-Hedlund Theorem, $g$ has no ultimately periodic natural coding. Then, by \cite[Lemma 3.1]{BP}, we have that $\omega(x)$ is a set with infinite cardinality. By \cite[Theorem 3.5]{BP}, $g$ admits a non-atomic $g$-invariant Borel probability measure. Proceeding as in the proof of \cite[Corollary 2.2]{BP2}, we conclude that $g$ is topologically semiconjugate to a $r$-IET $T:I\to I$ with $r\le n$. The fact that $g$ has only natural codings with maximal complexity implies that $r=n$ and that $T$ satisfies the i.d.o.c. 
  
$(ii)\implies (i)$. 
Denote by $\mathcal{D}(g)=\{x_1,\ldots,x_{n-1}\}$ and $\mathcal{D}(T)=\{y_1,\ldots,y_{n-1}\}$  the discontinuity sets of $g$ and $T$, respectively. Let $h:I\to I$ denote the topological semiconjugacy, thus $\mathcal{D}(T)\subseteq h\big(\mathcal{D}(g)\big) $. Since $\mathcal{D}(g)$ and $\mathcal{D}(T)$ have the same cardinality $n$ and since $h$ is non-decreasing, we have that $y_i=h(x_i)$ for every $1\le i\le n-1$. In particular, $h\big([x_{i-1},x_i)\big)=[y_{i-1},y_i)$ for all
$1\le i\le n$, where $x_0=y_0=0$ and $x_n=y_n=1$. This means that the natural $f$-coding $\theta$ of any $x\in I$ equals the natural $T$-coding $\theta$ of $h(x)$. Moreover, since $T$ is an $n$-IET satisfying the i.d.o.c., we have that the complexity function of $\theta$ by $T$ is $p_\theta(k)= (n-1)k+1$ for every $k\ge 1$ (see \cite[Lemma 6.2]{BP}).
\end{proof}

\section{Proof of Theorem \ref{thm1}}

Throughout this section, let $f:I\to I$, $G=I{\setminus}f(I) $, $\mathcal{N}\subset (-\inf G,1-\sup G)$, $\delta\in\mathcal{N}$,
$f_{\delta}:I\to I$ and $G^{(\delta)}=I{\setminus}f_{\delta}(I)$ be as in the statement of Theorem \ref{thm1}.  The gap sets
$G$ and $G^{(\delta)}$ are the union of $1\le m\le n-1$ connected components:
\begin{equation}\label{split}
G=G^{(1)}\cup\ldots\cup G^{(m)},\quad G_\delta=G^{(1)}_\delta\cup \ldots\cup G^{(m)}_\delta.
\end{equation}
In what follows, we will keep $\delta\in\mathcal{N}$ fixed. We will use $\vert\cdot\vert$ to denote interval length.

\begin{lemma}\label{oi} There exist infinite words $w^{(j)}=w_0^{(j)} w_1^{(j)}\ldots$, 
 $1\le j\le m$, over the alphabet $\{0,1\}$
 such that $p_{w^{(j)}}(k)\le (n-1)k+2$ for every $k\ge 1$ and $1\le j\le m$, and
\begin{equation}\label{eq31}
\delta=-d_1+\sum_{j=1}^m \sum_{i= 0}^\infty w_i^{(j)}b^{-i}\left\vert G^{(j)}\right\vert.
\end{equation}
\end{lemma}
\begin{proof} Let $J=[0,d_1+\delta)$. Let $q_j$ be the middle-point of $G_\delta^{(j)}$ and
$w^{(j)}=w^{(j)}_0 w^{(j)}_1\ldots$ be the infinity binary word defined by
$w_i^{(j)}=\rchi_J\big(f_{\delta}^{i}(q_j)\big)$. Since $\delta\in(-\inf G,1-\sup G)$,
we have that $f_{\delta}^{-1}(0)\in \mathcal{D}(f_\delta)$. This combined with the fact that $d_1+\delta=f_{\delta}(0)$ 
yield $f_\delta^{-2}(d_1+\delta)\in \mathcal{D}(f_{\delta})$.
Applying Lemma \ref{lem2.1} with $g=f_\delta$ we reach $\mathcal{D}(f_\delta\big)\cap\bigcup_{i= 0}^\infty f_\delta^i\big(G_\delta)=\emptyset$. In this way, using the injectivity of $f_{\delta}$ and (\ref{split}), we reach
$d_1+\delta\not\in \bigcup_{j=1}^m\bigcup_{i=0}^\infty f_\delta^i\big(G_\delta^{(j)}\big)$. This together with Lemma \ref{lem2.2} with $g=f_{\delta}$ yield
\begin{equation}\label{eq32}
d_1+\delta= \vert J\vert=\sum_{j=1}^m \sum_{i=0}^\infty w_i^{(j)} \left\vert f^i\left(G_\delta^{(j)}\right)\right\vert=\sum_{j=1}^m \sum_{i=0}^\infty w_{i}^{(j)} b^{-i}\left\vert G_\delta^{(j)}\right\vert.
\end{equation}
Since $\delta\in (-\inf G,1-\sup G)$, we have that $G+\delta\subset I$, which yields $\{G+\delta\}=G+\delta$. Hence,
$$ G_\delta=I{\setminus} R_{\delta}\big(f(I)\big)=R_\delta\big(I{\setminus}f(I)\big)=R_\delta(G)=\{G+\delta\}=G+\delta.
$$
In this way, by (\ref{split}), $G_\delta^{(j)}=G^{(j)}+\delta$ for all $1\le j\le m$. This shows that
\begin{equation}\label{eq33}
\left\vert G_\delta^{(j)}\right\vert=\left\vert G^{(j)}\right\vert\quad\textrm{for all}\quad 1\le j\le m.
\end{equation}
Replacing (\ref{eq33}) in (\ref{eq32}) yield (\ref{eq31}).

Let us now prove that $p_{w^{(j)}}(k)\le (n-1)k+1$ for all $k\ge 1$. Let $\mathscr{P}_{\delta}=\{I_\delta^{(1)},\ldots,I_\delta^{(n)}\}$ be the partition into intervals associated with the injective $n$-PC $f_{\delta}$. Since $d_1+\delta=f_\delta(0)$, we have that $f_{\delta}^{-1}\big(J\big)$ is the union of finitely many members of   $\mathscr{P}_{\delta}$. Let $\eta:\{1,2,\ldots,n\}\to \{0,1\}$ the map defined by $\eta(i)=1$ if $f_\delta(I_\delta^{(i)})\subset J$ and $\eta(i)=0$ if $f_\delta(I_\delta^{(i)})\subset I{\setminus} J$. As before, let $q_j$ be the middle-point of $G_{\delta}^{(j)}$ and let
$\theta^{(j)}=\theta_0^{(j)}\theta_1^{(j)}\ldots$ denote the natural $f_{\delta}$-coding of $q_j$. Then
$w_{i+1}^{(j)}=\eta\big(\theta_i^{(j)}\big)$ for all $i\ge 0$ and $1\le j\le m$. Hence,
$w^{(j)}=w^{(j)}_0\eta \big(\theta^{(j)}_0\big) \eta \big(\theta^{(j)}_1\big)\ldots$. Writing
${w}^{(j)}=w^{(j)}_0v$, then $v=\eta \big(\theta^{(j)}_0\big) \eta \big(\theta^{(j)}_1\big)\ldots$.  In this way, for each $k\ge 1$, the map $\eta$ induces a surjective map from $L_k\big(\theta^{(j)}\big)$ onto $L_k\big(v\big)$, implying that
$$ p_v(k)=\#L_k(v)\le \#L_k\big(\theta^{(j)}\big)=p_{\theta^{(j)}}(k)=(n-1)k+1.
$$
Therefore,
$$ p_{w^{(j)}}(k)\le 1+p_v(k)=(n-1)k+2.$$
\end{proof}

In what follows, $\mathcal{A}=\{0,1,\ldots,b-1\}$. To each infinite word
$w=w_0 w_1\ldots$ over the alphabet $\mathcal{A}$, we can
associate naturally the real number $x$ in $[0, 1]$ whose representation in base $b$ is given by $w$.
More precisely, $x=\rho(w)$,  where  $\rho: \mathcal{A}^{\mathbb{N}}\to [0,1]$ is the function 
defined by
$$ \rho(w) =\sum_{i=0}^\infty w_i b^{-(i+1)}=b^{-1}\sum_{i=0}^\infty w_i b^{-i}.$$
The \textit{entropy} of an infinite word $w=w_0 w_1\ldots$ over the alphabet $\mathcal{A}$ is the number $E(w) \log b$, where
$$ E(w)=\lim_{k\to\infty}\dfrac{\log_b \big(p_w(k)\big)}{k}.
$$

We will need the following result.
\begin{theorem}[Mauduit-Moreira {\cite[Proposition 1]{MM}}]\label{tmi} For each integer $b\ge 2$, the set $ \left\{\rho(w): w\in\mathcal{A}^{\mathbb{N}}\,\,\textrm{and}\,\,E(w)=0 \right\} 
$
has Hausdorff dimension $0$.
\end{theorem}

Given two  infinite words $u=u_0u_1\ldots$ and $v=v_0v_1\ldots$ over the alphabet $\{0,1\}$, we denote by $u+v$ the infinite word over the alphabet $\{0,1,2\}$ whose $i$-th letter is \linebreak $u_i+v_i$. Likewise, the sum of $m$ infinite binary words is an infinite word over the alphabet $\{0,1,\ldots,m\}\subseteq \mathcal{A}$ because $m\le n-1\le b-1$, where $m$ is the number of connected components of the gap set $G$ (which has the same number of connected components as $G_\delta$).

\begin{lemma}\label{u} Let $w^{(1)},\ldots, w^{(m)}$ be the binary words in the statement of Lemma \ref{oi}. Then, $w=w^{(1)}+\ldots+w^{(m)}$ is an infinite word
over the alphabet $\{0,1,\ldots,b-1\}$ with $E(w)=0$. 
\end{lemma}
\begin{proof} Since $m\le b-1$ and also because each word $w^{(j)}$ is binary,
we have that $w=w^{(1)}+\ldots+w^{(m)}$ is an infinite word over the alphabet $\{0,1,\ldots,b-1\}$. The operation of adding $m$ words componentwisely yields a surjective map from
$L_k(w^{(1)})\times \ldots\times L_k(w^{(m)})$ onto  $L_k(w)$. Then, by Lemma \ref{oi}, for every $k\ge 1$,
$$p_{w}(k)\le \prod_{j=1}^m p_{w^{(j)}}(k)\le \left[(n-1) k +2 \right]^m.$$
Moreover,
$$ E(w)=\lim_{k\to\infty}\dfrac{\log_b \big(p_{w}(k)\big)}{k}\le m\dfrac{\log_b \big[(n-1)k+2 \big]}{k}=0.
$$
\end{proof}

\begin{proof}[Proof of Theorem \ref{thm1}] Let $\mathcal{A}=\{0,1,\ldots,b-1\}$ and
$$\mathcal{N}_1=\left\{\sum_{j=1}^m \sum_{i= 0}^\infty w_i^{(j)}b^{-i}: w^{(j)}=w_0^{(j)}w_1^{(j)}\ldots\in\{0,1\}^\mathbb{N}\,\,\mathrm{and}\,\, p_{w^{(j)}}(k)\le (n-1)k+2, \forall k,j
\right\}$$
It follows from Lemma \ref{u} that
$$\mathcal{N}_1\subset\left\{\sum_{i= 0}^\infty w_i b^{-i}: w=w_0w_1\ldots\in\mathcal{A}^\mathbb{N}\,\,\mathrm{and}\,\, E(w)=0
\right\}.$$
In other words,
$\mathcal{N}_1\subset\left\{b\rho(w):w\in\mathcal{A}^{\mathbb{N}}\,\,\,\mathrm{and}\,\,\, E(w)=0\right\}.$
By Theorem \ref{tmi}, the set $\mathcal{N}_1$ has Hausdorff dimension $0$. By Lemma
\ref{oi},  since $\vert G^{(j)}\vert=L$ for all $1\le j\le m$ (see Definition \ref{def1}), we have that
$\mathcal{N}\subset -d_1+L\mathcal{N}_1$. In this way, $\mathcal{N}$ has Hausdorff dimension $0$. 
\end{proof}

\end{document}